\documentclass{amsart}
\usepackage{tikz}
\usepackage{xcolor}
\usepackage{amssymb,latexsym,amsmath,extarrows}
\usepackage{graphicx,mathrsfs,comment}
\usepackage{hyperref,url}

\usepackage{amstext}
\usepackage{bbm}

\numberwithin{equation}{section}

\newtheorem{theorem}{Theorem}[section]
\newtheorem{lemma}[theorem]{Lemma}

\newtheorem{remark}[theorem]{Remark}

\newtheorem{corollary}[theorem]{Corollary}

\newtheorem{conjecture}[theorem]{Conjecture}

\newcommand{\dint}{\displaystyle\int}

\newcommand{\R}{\mathbb{R}}

\newcommand{\N}{\mathcal{N}}

\begin{document}

\title[]{An extended Vinogradov's mean value theorem}

\author{Changkeun Oh}\address{ Changkeun Oh\\ Department of Mathematical Sciences and RIM, Seoul National University, Republic of Korea} \email{changkeun@snu.ac.kr}

\author{Kiseok Yeon} \address{ Kiseok Yeon\\  Department of Mathematics, University of California, Davis, United States}\email{kyeon@ucdavis.edu}

\begin{abstract}

In this paper, we provide novel mean value estimates for exponential sums related to the extended main conjecture of Vinogradov's mean value theorem, by developing the Hardy-Littlewood circle method together with a refined shifting variables argument.

Let $d\geq 2$ be a natural number and $\boldsymbol{\alpha}=(\alpha_d,\ldots, \alpha_1)\in \R^d.$ Define the exponential sum
\begin{equation*}
    f_d(\boldsymbol{\alpha};N):=\sum_{1 \leq n \leq N}e(\alpha_d n^d + \cdots+ \alpha_1 n).
\end{equation*}
For $p>0$,  consider mean values of the exponential sums
\begin{equation*}
    \mathcal{I}_{p,d}(u;N):=\dint_{[0,1)\times [0,N^{-u})\times [0,1)^{d-2}}|f_d(\boldsymbol{\alpha};N)|^pd\boldsymbol{\alpha},
\end{equation*} where we wrote $d\boldsymbol{\alpha}=d\alpha_1 d\alpha_2\cdots d\alpha_{d-1}d\alpha_d.$ By making use of the aforementioned tools, we obtain the sharp upper bound for $\mathcal{I}_{p,d}(u;N)$, for $d=2,3$ and $0<u\leq 1$. Furthermore, for $d \geq 4$, we obtain analogous results depending on a small cap decoupling inequality for the moment curves in $\mathbb{R}^d.$

\end{abstract}

\maketitle


\section{Introduction}

 In this paper, we provide novel mean value estimates related to the extended main conjecture of Vinogradov's mean value theorem \cite[Conjecture 8.2]{MR4571633}.

\begin{conjecture}[\cite{MR4571633}]\label{conjecture 1.1}
    Suppose that $d\in \mathbb{N}$, $\boldsymbol{\alpha}=(\alpha_d,\ldots, \alpha_1)\in \R^d$ and $\mathfrak{D}\subseteq [0,1)^d$ is measurable. Then, whenever $s$ is a positive number and 
    \begin{equation*}
        \textrm{mes}(\mathfrak{D})\gg N^{1-d(d+1)/4},
    \end{equation*}
    one has 
    \begin{equation}\label{main}
        \int_{\mathfrak{D}}\biggl|\sum_{1 \leq n \leq N}e\biggl(\sum_{1\leq i\leq d}\alpha_in^i\biggr)\biggr|^{2s}d\boldsymbol{\alpha}\ll N^{\epsilon}(N^{s}\text{mes}(\mathfrak{D})+N^{2s-d(d+1)/2}),
    \end{equation}
    where $d\boldsymbol{\alpha}=d\alpha_1 d\alpha_2\cdots d\alpha_{d-1}d\alpha_d.$
\end{conjecture}

 When $\mathfrak{D}=[0,1)^d$ with $d\geq 2$, this conjecture becomes Vinogradov's mean value theorem.  This has been completely resolved by Bourgain, Demeter and Guth \cite{MR3548534} making use of decoupling inequalities for the moment curve, and by Wooley \cite{MR3479572, MR3938716} developing efficient congruencing arguments. 

Beyond $\mathfrak{D}=[0,1)^d$, Demeter, Guth and Wang \cite{MR4153908} initially resolved Conjecture $\ref{conjecture 1.1}$
with  $\mathfrak{D}=[0,N^{-u})\times[0,1)^{d-1}$ for $d=2,3$ and  certain ranges of $u$, by introducing small cap decouplings for the moment curve in $\R^2$ and $\R^3$ (see explanation following Conjecture $\ref{06.29.conj14}$, for the detailed history on small cap decouplings). As an application of the result associated with the set $\mathfrak{D}=[0,N^{-u})\times[0,1)^{d-1}$, one could make improvement on  fourth derivative estimates for exponential sums, which delivers a neat bound for the Lindel\"of $\mu(\sigma)$ function associated the Riemann zeta function (see [\cite{MR4153908}, Appendix]). When $d=4,$ Demeter \cite{MR4491437} obtained the sharp bound for the mean value in the left hand side of $(\ref{main})$,  with $2s=12$ and $\mathfrak{D} =[0,N^{-a}] \times [0,N^{-b}] \times [0,1] \times [0,1]$ satisfying $a+b=3$ with $1 \leq a \leq 2$. In \cite{MR4491437}, Demeter obtained a more general result, specifically, that is the sharp bound for mean values of exponential sums related to a non-degenerate curve $(n,n^2,\phi_3(n),\phi_4(n))$ with the domain $[0,N^{-a}] \times [0,N^{-b}] \times [0,1] \times [0,1]$ satisfying $a+b=3$ with $1 \leq a \leq 2$. Prior to this work, Bourgain \cite{MR3556291} deduced the same sharp bound for the case  $(a,b)=(2,1)$, which makes an improvement on Lindel\"of hypothesis. Furthermore, for $d=5$, we refer to \cite{MR4537326} for a related work. For related works on other shapes of local mean value estimate, see also \cite{MR4814698}, \cite[section 4]{MR4542717}, \cite{MR4310303}.

In this paper, we consider the set $\mathfrak{D}=[0,1)\times [0,N^{-u})\times[0,1)^{d-2}$ with $d\geq2$, and confirm Conjecture $\ref{conjecture 1.1}$ for $d=2,3$ and $0<u\leq 1$ (see Theorem $\ref{05.28.thm111}$ and $\ref{05.28.thm11}$ below). We notice here that for the case $d=2$, we prove more than Conjecture $\ref{conjecture 1.1}$, since the measure of $\mathfrak{D}=[0,1)\times [0,N^{-u})\times[0,1)^{d-2}$ can be smaller than $N^{1-d(d+1)/4}=N^{-1/2}$, in the range $1/2< u\leq 1$. Also, for higher degrees $d\geq 4$ and $0<u\leq 1,$ we confirm Conjecture $\ref{conjecture 1.1}$ conditionally (see Theorem $\ref{thm1.6}$ below). Furthermore, even for the range $u>1$, we provide the range of $s$ so that the inequality $(\ref{main})$ holds, superior to the trivial range obtained from Vinogradov's mean value theorem. To deal with the set $\mathfrak{D}=[0,1)\times [0,N^{-u})\times[0,1)^{d-2}$, we make use of tools in number theory. Specifically, we develop shifting variables argument initially introduced by Wooley \cite{MR2913181}, and develop the Hardy-Littlewood circle method adapted for the mean value \begin{equation*}
      \int_{[0,1)\times [0,N^{-u})\times[0,1)^{d-2}}\biggl|\sum_{1 \leq n \leq N}e\biggl(\sum_{1\leq i\leq d}\alpha_in^i\biggr)\biggr|^{2s}d\boldsymbol{\alpha}.
\end{equation*} 
Our hope is that techniques described here may be useful in exploring Conjecture $\ref{conjecture 1.1}$ in general. 

\bigskip

\bigskip

To facilitate the statement of theorems, we provide some definitions.
Let $d\geq 2$ be a natural number and $\boldsymbol{\alpha}=(\alpha_d,\ldots, \alpha_1)\in \R^d.$ Define the exponential sum
\begin{equation*}
    f_d(\boldsymbol{\alpha};N):=\sum_{1 \leq n \leq N}e(\alpha_d n^d + \cdots+ \alpha_1 n).
\end{equation*}
For $p>0$,  consider mean values of the exponential sums
\begin{equation*}
    \mathcal{I}_{p,d}(u;N):=\dint_{[0,1)\times [0,N^{-u})\times [0,1)^{d-2}}|f_d(\boldsymbol{\alpha};N)|^pd\boldsymbol{\alpha},
\end{equation*} where, here and throughout, we write $d\boldsymbol{\alpha}=d\alpha_1 d\alpha_2\cdots d\alpha_{d-1}d\alpha_d.$

\bigskip

The following two theorems provide the sharp bounds for $\mathcal{I}_{p,d}(u;N)$ for $d=2,3$ and $0< u\leq 1.$ Even for $u>1$ and $d=3$, Theorem $\ref{05.28.thm11}$ provides the range of $p$ such  that $ \mathcal{I}_{p,3}(u;N)$ has the sharp upper bound,  superior to the trivial range of $p$ obtained from Vinogradov's mean value theorem.

\begin{theorem}\label{05.28.thm111}
 For $0< u\leq 1$ and $p>0$, one has
\begin{equation*}
   \mathcal{I}_{p,2}(u;N)\ll N^{\epsilon}(N^{p-3}+ N^{p/2-u}). 
\end{equation*}
\end{theorem}

\bigskip

\begin{theorem}\label{05.28.thm11}
We have the following:

(i) For $0< u\leq 1$ and $p>0$, one has
\begin{equation*}
   \mathcal{I}_{p,3}(u;N)\ll N^{\epsilon}(N^{p-6}+ N^{p/2-u}). 
\end{equation*}

(ii) For  $1< u\leq 2$ and $p\geq 12-6/(4-u),$ one has   
   $$\mathcal{I}_{p,3}(u;N)\ll  N^{p-6+\epsilon}.$$

\end{theorem}
As for comparison with results available hitherto, it follows by Vinogradov's mean value theorem  that whenever $p\geq d(d+1)$, one has
\begin{equation*}
    \mathcal{I}_{p,d}(u;N)\ll  \dint_{[0,1)^d}|f_d(\boldsymbol{\alpha};N)|^pd\boldsymbol{\alpha}\ll  N^{p-d(d+1)/2+\epsilon}.
\end{equation*}
Furthermore, whenever $p\leq d(d-1),$ it follows by \cite[Theorem 1.1]{MR3938716} that
\begin{equation*}
\begin{aligned}
     \mathcal{I}_{p,d}(u;N)&\ll N^{-u}\sup_{\alpha_{d-1}\in [0,N^{-u}]}  \dint_{[0,1)^{d-1}}|f_d(\boldsymbol{\alpha};N)|^pd\alpha_d d\alpha_{d-2} d\alpha_{d-3}\cdots d\alpha_1\\
     &\ll  N^{p/2-u+\epsilon}.
\end{aligned}
\end{equation*}
Hence, one notices here that the difficulties may occur from that obtaining the sharp bound in the range $d(d-1)<p<d(d+1).$ We see by Theorem $\ref{05.28.thm111}$ and $\ref{05.28.thm11}$ that when $d=2,3$ and $0<u\leq 1$, we obtain the sharp upper bound in the range $d(d-1)<p<d(d+1),$ and that when $d=3$ and $1<u\leq 2$, we derive the sharp bound in the range $p\geq 12-6/(4-u).$

\bigskip

\begin{corollary}\label{coro1.4}
The number of integer solutions $n_i\ (1\leq i\leq 10)$ with $1\leq n_i\leq N$, satisfying the system 
   \begin{equation}\label{1.1.1}
    \begin{aligned}
       &\sum_{1\leq i
       \leq 5}(n_i^3-n_{5+i}^3)=0,\\
            \biggl|&\sum_{1\leq i\leq 5}(n_i^2-n_{5+i}^2)\biggr|\leq N\\
          &\sum_{1\leq i\leq 5}(n_i-n_{5+i})=0
        \end{aligned},
    \end{equation}
    is $O(N^{5+\epsilon}).$
\end{corollary}
Let us show that Corollary $\ref{coro1.4}$ is essentially optimal. Take diagonal solutions
\begin{equation}
    n_i=n_{5+i}, \;\;\; 1 \leq i \leq 5.
\end{equation}
Since $1 \leq n_i \leq N$, this already gives $N^5$ solutions of \eqref{1.1.1}. 

Corollary \ref{coro1.4} can be thought of as an intermediate step toward solving the mean value estimate for Weyl sums for the pair $(n^3,n)$. To see this, for given $0 \leq \alpha \leq 2$ and $m \geq 1$, let us consider the number of integer solutions $n_i$ with $1 \leq n_i \leq N$, satisfying the system
\begin{equation}
    \begin{aligned}
       &\sum_{1\leq i
       \leq m}(n_i^3-n_{m+i}^3)=0,\\
            \biggl|&\sum_{1\leq i\leq m}(n_i^2-n_{m+i}^2)\biggr|\leq N^{\alpha},\\
          &\sum_{1\leq i\leq m}(n_i-n_{m+i})=0.
        \end{aligned}
    \end{equation}
The case $\alpha=0$ corresponds to the  main conjecture in Vinogradov's mean value theorem for degree three, which was proved by \cite{MR3479572, MR3548534}. The case $\alpha=2$ relates to the conjectural mean value estimate for Weyl sums associated with $(n^3,n)$, which remains open. It does not look clear how to prove the conjecture by using the current decoupling techniques. The case $\alpha=1$ corresponds to Corollary \ref{coro1.4}, which gives a sharp solution counting estimate for any $m \geq 1$. So this corollary can be viewed as an intermediate step toward solving the conjecture for the pair $(n^3,n)$.

\begin{proof}[Proof of Corollary 1.4]
It follows by \cite[Lemma 2.1]{MR1002452}  with $x_i=\alpha_i$, $w_k(u)=u^k$, $K=3$, $M=5,$ $\mathscr{A}=[1,N]$, $\delta_2=N$, $\delta_1=\delta_3=1$ that the number of integer solutions satisfying $(\ref{1.1.1})$ is asymptotically same as $N$ times $\mathcal{I}_{10,3}(1;N)$.
Meanwhile, we find by Theorem $\ref{05.28.thm11}$ with $u=1$ and $p=10$ that
    \begin{equation*}
        \mathcal{I}_{10,3}(1;N)\ll N^{4+\epsilon}.
    \end{equation*}
    Therefore, this complete the proof of Corollary $\ref{coro1.4}.$
\end{proof}

\bigskip



To prove Theorem $\ref{05.28.thm111}$ and $\ref{05.28.thm11}$, we shall prove a theorem more general than these. In order to describe this theorem, we require the following conjecture.

\begin{conjecture}\label{06.29.conj14}
Let $d \geq 2$. For $0 \leq u \leq d-1$ and $p$ an even number, one has
\begin{equation}\label{06.29.11}
      \dint_{[0,N^{-u})\times [0,1)^{d-1}}|f_d(\boldsymbol{\alpha};N)|^pd\boldsymbol{\alpha} \ll N^{\epsilon}\left(N^{p-\frac{d(d+1)}{2}}+N^{p/2-u}\right).
    \end{equation}
  
\end{conjecture}

Conjecture \ref{06.29.conj14} is first formulated in \cite{MR4153908} (see Conjecture 2.5 therein). More precisely, they made a stronger conjecture that \eqref{06.29.11} is true for $d \geq 2$, $0 \leq u \leq d-1$, and $p$ any real numbers with $2 \leq p < \infty$. To attack the conjecture they made, they rather conjectured a small cap decoupling inequality for the parabola in $\mathbb{R}^2$, which easily imply \eqref{06.29.11} for $d=2$, $0 \leq u \leq 1$, and $2 \leq p < \infty$. In \cite{MR4153908}, they proved
a small cap decoupling for the parabola, which implies \eqref{06.29.11} for $d=2$, $0 \leq u \leq 1$, and  $2 \leq p <\infty$. In addition, they proved \eqref{06.29.11} for the case $d=3$, $0 \leq u \leq \frac32$, and $2 \leq p <\infty$. It took a couple of years to fully solve the conjecture of \cite{MR4153908} for $d=3$. To explain, a different proof of a small cap decoupling for the parabola is introduced by \cite{MR4603641}, which is inspired by a new proof of a decoupling for the parabola by \cite{MR4721026}. This method is developed by \cite{guth2022smallcapdecouplingmoment} to prove a small cap decoupling for the moment curve in $\mathbb{R}^3$, and they obtained \eqref{06.29.11} for the case that $d=3$, $0 \leq u \leq 2$, and $2 \leq p <\infty$. Hence, Conjecture \ref{06.29.conj14} is obviously verified by above, for $d=2,3$.  We give a remark that a third proof of a small cap decoupling for the parabola is given by \cite{maldague2022amplitudedependentwaveenvelope}.

\begin{theorem}\label{thm1.6}
 Assume that Conjecture 1.5 is true. Then, for $d\geq 2,$ we have the following:

(i) For $0< u\leq 1$ and $p>0$, one has
\begin{equation*}
   \mathcal{I}_{p,d}(u;N)\ll N^{\epsilon}(N^{p-d(d+1)/2}+ N^{p/2-u}). 
\end{equation*}

(ii)  For $1<u\leq d-1$ and $p\geq d(d+1)-\frac{2d}{d+1-u}$, one has   
   $$\mathcal{I}_{p,d}(u;N)\ll  N^{p-d(d+1)/2+\epsilon}.$$
\end{theorem}

\bigskip

\begin{proof}[Proof of Theorem 1.2 and 1.3]
We see by explanation preamble to Theorem $\ref{thm1.6}$ that Conjecture $\ref{06.29.conj14}$ is true for $d=2,3$. Hence, one finds that Theorem $\ref{thm1.6}$ implies Theorem $\ref{05.28.thm111}$ and $\ref{05.28.thm11}.$
\end{proof}

\addtocontents{toc}{\protect\setcounter{tocdepth}{0}}
\section*{Stucture of the paper and notation}
It remains to prove Theorem $\ref{thm1.6}.$ Hence, the vast majority of this paper is devoted to lemmas for the proof of  Theorem $\ref{thm1.6}.$ In section $\ref{sec2}$, we provide Lemma $\ref{lem2.1}$ and $\ref{lem2.2}$ deriving upper bounds for $\mathcal{I}_{p,d}(u;N)$ in terms of other types of mean values of exponential sums. In section \ref{sec3}, we define major and minor arcs. By using these dissections, we provide Lemma $\ref{lem3.1}$ and $\ref{lem3.2}$ on bounds for mean values of exponential sums over major and minor arcs, respectively. Furthermore, we provide the proof of Theorem $\ref{thm1.6}$ at the end of section $\ref{sec3}$, by using Lemma $\ref{lem3.1}$ and $\ref{lem3.2}.$ In section $\ref{sec4}$ and $\ref{sec5}$, we prove Lemma $\ref{lem3.1}$ and $\ref{lem3.2},$ respectively.

\addtocontents{toc}{\protect\setcounter{tocdepth}{2}}

\section*{Acknowledgement}
Changkeun Oh
was supported by the New Faculty Startup Fund from Seoul National University, the POSCO Science Fellowship of POSCO TJ Park Foundation, and the National Research Foundation of Korea (NRF) grant funded by the Korea government (MSIT)  RS-2024-00341891. Kiseok Yeon was supported by the KAP allocation from the University of California, Davis. The authors would like to thank Larry Guth and Trevor Wooley for valuable comments.

\section{Preliminary manoeuvre}\label{sec2}

In this section, we provide upper bounds for the mean value $\mathcal{I}_{p,d}(u;N),$ in terms of different shape of mean values of exponential sums. These mean values of exponential sums include an extra exponential sum aside from $f_d(\boldsymbol{\alpha};N),$ which delivers  bounds superior to the trivial bounds on  $\mathcal{I}_{p,d}(u;N),$ eventually.

To describe the first lemma in this section, we introduce some definitions. Let  $\varphi: \R\rightarrow \R^+$ be a smooth function supported on $[-2,2]$ and $\varphi(x)\geq 1$ on $x\in [-1,1].$ For $A>0$ and $\beta\in \R$, define 
    \begin{equation}\label{2.0}
        \Psi_A(\beta):=\Psi_A(\beta;N)=\sum_{j\in \mathbb{Z}}\varphi\left(\frac{\beta+j}{N^{-A}}\right),
    \end{equation}
    which is clearly $1$-periodic, smooth and supported on $\|\beta\|_{\R/\mathbb{Z}}\leq 2N^{-A}$.
Furthermore, for $p>0$ and $u>0$ define 
\begin{equation}\label{2.32.3}
S_p(u)=\int_{[0,1)^d}|f_d(\boldsymbol{\alpha};2N)|^p  \cdot \sum_{1\leq y\leq N}\Psi_u(\alpha_{d-1}-dy\alpha_d) d\boldsymbol{\alpha}.
\end{equation}
The following lemma provides an upper bound for $\mathcal{I}_{p,d}(u;N)$ in terms of $S_p(u)$ for the case $d\geq3.$  For the case $d=2$, see Remark $\ref{remark2.3}.$ 

\begin{lemma}\label{lem2.1}
   For $p>1$ and $d\geq 3$, we have 
    \begin{equation*}
    \begin{aligned}
       \mathcal{I}_{p,d}(u,N)\ll N^{-1}(\log N)^{p} S_p(u).
    \end{aligned}
    \end{equation*}
\end{lemma}

\bigskip

We would like to give a remark. The discussion below is not used in the subsequent proof.
    By applying the Poisson summation formula to $\Psi_u(\alpha_{d-1}-dy\alpha_d)$ in $S_p(u)$, we have
    \begin{equation}\label{2.42.42.4}
S_p(u)=N^{-u}\int_{[0,1)^d}|f_d(\boldsymbol{\alpha};2N)|^p  \cdot \sum_{1\leq y\leq N}\sum_{j\in \mathbb{Z}}\hat{\varphi}\left(\frac{j}{N^{u}}\right)e((\alpha_{d-1}-dy\alpha_d)j) d\boldsymbol{\alpha}.
    \end{equation}
    By the definition of $\varphi$, one has $\hat{\varphi}\left(\frac{j}{N^{u}}\right)\ll_{\epsilon,K} |j|^{-K}$ for all $K>0$, whenever $|j|\geq N^{u+\epsilon}$. Hence, one deduces that
\begin{equation*}
    S_p(u)\ll N^{-u}\int_{[0,1)^d}|f_d(\boldsymbol{\alpha};2N)|^p  \cdot \sum_{|j|\leq N^{u+\epsilon}}\biggl|\sum_{1\leq y\leq N}e(-dyj\alpha_d)\biggr| d\boldsymbol{\alpha}.
\end{equation*}
Then, one infers from Lemma $\ref{lem2.1}$ that 
\begin{equation}\label{2.42.4}
\begin{aligned}
     &\mathcal{I}_{p,d}(u,N)\\
     &\ll N^{-u-1}(\log N)^p\int_{[0,1)^d}|f_d(\boldsymbol{\alpha};2N)|^p  \cdot \sum_{|j|\leq N^{u+\epsilon}}\biggl|\sum_{1\leq y\leq N}e(-dyj\alpha_d)\biggr| d\boldsymbol{\alpha}.
\end{aligned}
\end{equation}
By making use of the shifting variables argument previously introduced by Wooley \cite{MR2913181} combined with \cite[Lemma 2.1]{MR1002452}, the inequality $(\ref{2.42.4})$ can be established when $p$ is an even number. In the proof of Lemma $\ref{lem2.1}$, we introduce a refined shifting variables argument. We emphasize that this is applicable to $L^p$-norm of exponential sums, without confining $p$ to even numbers. The shifting variables argument applied hitherto has widely used in making improvements on many problems in number theory as listed following. In \cite{MR2913181}, Wooley improved on the number of variables required to establish the asymptotic formula in Waring's problem. By using such argument, Wooley \cite{MR3431575} establishes an essentially optimal estimate for ninth moment of exponential sum having argument $\alpha n^3+\beta n$ . In general, Wooley \cite[Theorem 14.4]{MR3938716} recorded mean value estimates for exponential sums having argument $\sum_{1\leq i\leq t}\alpha_{k_t}n^{k_t}$ with $k_1> k_2>\cdots>k_t$ and $k_1-1>k_2$. The second author \cite[Theorem 1.3 and 1.4]{MR4686663} generalized to that having argument $\sum_{1\leq i\leq t}\alpha_{k_t}n^{k_t}$ with $k_1> k_2> \cdots> k_t$. Furthermore, Brandes and Hughes \cite{MR4510129} proved that inhomogeneous Vinogradov's system has approximately fewer integer solution in the subcritical range, than its homogenous counter part. 
Wooley \cite{MR4549146} derived the asymptotic formula for the number of integer solutions for the inhomogeneous cubic Vinogradov's system in the critical case, and extended this to higher degree cases under the assumption of the extended main conjecture of Vinogradov's mean value theorem \cite{MR4571633}.


\bigskip

\begin{proof}
    It follows by the definition ($\ref{2.0}$) of $\Psi_u$ that 
    \begin{equation}\label{2.2}
        \mathcal{I}_{p,d}(u,N)\leq \int_{[0,1)^d}|f_d(\boldsymbol{\alpha};N)|^p\Psi_u(\alpha_{d-1})d\boldsymbol{\alpha}.
    \end{equation}
    Let $K(\gamma)=\sum_{1\leq z\leq N}e(\gamma z)$. Write $\Omega(n;\boldsymbol{\alpha})=\alpha_dn^d+\cdots+\alpha_1n.$ Then, for $y\in [1,N]\cap \mathbb{Z}$, we observe by orthogonality that 
    \begin{equation}\label{2.3}
    \begin{aligned}
        f_d(\boldsymbol{\alpha};N)&=\sum_{y<n\leq N+y}e(\Omega(n-y;\boldsymbol{\alpha}))\\
        &=\int_0^1\sum_{1\leq n\leq 2N}e(\Omega(n-y;\boldsymbol{\alpha})+\gamma(n-y))K(\gamma)d\gamma.
    \end{aligned}
    \end{equation}
   
   Note by applying the Binomial expansion that
    \begin{equation}\label{2.4}
        \begin{aligned}
            \Omega(n-y;\boldsymbol{\alpha})= \sum_{l=1}^d\sum_{i=0}^l\left(\binom{l}{i}(-y)^{l-i}n^i\right)\alpha_l= \sum_{i=0}^dc_in^i,
        \end{aligned}
    \end{equation}
    where $c_i:=c_i(\boldsymbol{\alpha},y)=\sum_{l=i}^d\binom{l}{i}(-y)^{l-i}\alpha_l$ with $\alpha_0=0.$
Write 
\begin{equation}\label{07.02.28}
\mathfrak{c}:=\mathfrak{c}(\boldsymbol{\alpha},y)=(c_d,c_{d-1},\ldots,c_1)\in \R^d.    
\end{equation}
On substituting $(\ref{2.4})$  into $(\ref{2.3})$, we have
\begin{equation}\label{2.5}
f_d(\boldsymbol{\alpha};N)=\int_0^1 e(c_0-\gamma y)f_d(\mathfrak{c};2N,\gamma)K(\gamma)d\gamma,
\end{equation}
where 
$$f_d(\mathfrak{c};2N,\gamma)=\sum_{1\leq n\leq 2N}e(c_d n^d+\cdots+c_1n+\gamma n).$$
By substituting $(\ref{2.5})$ into $(\ref{2.2})$, we find that 
\begin{equation*}
    \mathcal{I}_{p,d}(u;N)\leq \int_{[0,1)^d}\biggl|\int_0^1 e(c_0-\gamma y)f_d(\mathfrak{c};2N,\gamma)K(\gamma)d\gamma\biggr|^p\Psi_u(\alpha_{d-1})d\boldsymbol{\alpha}.
\end{equation*}
By applying the H\"older's inequality together with the fact that $\Psi_u\geq0$, one has
\begin{equation}\label{2.6}
    \begin{aligned}
      &\mathcal{I}_{p,d}(u;N)\\
      &\leq \int_{[0,1)^d}\biggl(\int_0^1|f_d(\mathfrak{c};2N,\gamma_1)|^p|K(\gamma_1)|d\gamma_1 \biggr)\Psi_u(\alpha_{d-1}) d\boldsymbol{\alpha}\cdot \biggl(\int_0^1|K(\gamma_2)|d\gamma_2\biggr)^{p-1}\\
      &\ll (\log N)^{p-1}\int_0^1\biggl(\int_{[0,1)^d}|f_d(\mathfrak{c};2N,\gamma)|^p\Psi_u(\alpha_{d-1})d\boldsymbol{\alpha} \biggr)|K(\gamma)|d\gamma
    \end{aligned}
\end{equation}
where we have used $\int_0^1 |K(\gamma)|d\gamma\ll \log N.$ 

Note that the integrand in the last expression of $(\ref{2.6})$ is a $1$-periodic function in  $\alpha_i\ (1\leq i\leq d)$. Then,  by change of variables $\beta_i=c_i(\boldsymbol{\alpha},y)\ (1\leq i\leq d),$ and by writing $\boldsymbol{\beta}=(\beta_d,\beta_{d-1},\ldots,\beta_1),$ one finds that
\begin{equation}\label{2.7}
\begin{aligned}
    \mathcal{I}_{p,d}(u;N)\ll (\log N)^{p-1}\int_0^1\left(\int_{[0,1)^d}|f_d(\boldsymbol{\beta};2N,\gamma)|^p\Psi_u(\beta_{d-1}-dy\beta_d)d\boldsymbol{\beta}\right)|K(\gamma)|d\gamma,
\end{aligned}
\end{equation}
where we used the fact that the Jacobian determinant of the change of variables is $1.$ By summing over $1\leq y\leq N$ on both sides of $(\ref{2.7})$ and by change of variable $\beta_1\rightarrow \beta_1-\gamma$, we deduce that 
\begin{equation}\label{2.8}
    \begin{aligned}
        &\sum_{1\leq y\leq N}\mathcal{I}_{p,d}(u;N)\ll (\log N)^{p}S_p(u).
    \end{aligned}
\end{equation}
By dividing by $N$, this completes the proof of Lemma $\ref{lem2.1}.$
\end{proof}

\bigskip

To describe the next lemma, we require some definitions. Recall the definition ($\ref{2.0}$) of $\Psi_A(\cdot)$. 
Write $\boldsymbol{\alpha}=(\alpha_d,\tilde{\boldsymbol{\alpha}})\in \R^d$. Then, for given $p>0$, $u>0$ and $\epsilon>0$, we define 
\begin{equation}\label{Tpdef}
    T_p(u;\epsilon)=\int_{[0,1)^d}\int_0^1|f_d(\alpha_d-\beta,\tilde{\boldsymbol{\alpha}};2N)|^p \Psi_{u+1}(\beta) \sum_{1\leq y\leq N}\Psi_{u-\epsilon}(\alpha_{d-1}-dy\alpha_d)d\beta d\boldsymbol{\alpha}.
\end{equation}
The following lemma provides an upper bound for $\mathcal{I}_{p,d}(u;N)$ in terms of $ T_p(u;\epsilon).$
\begin{lemma}\label{lem2.2}
  For $\epsilon>0$, $p>1$ and $d\geq3,$ one has
 \begin{equation*}
    \begin{aligned}
       &\mathcal{I}_{p,d}(u;N)\ll N^u(\log N)^p T_p(u; \epsilon).
    \end{aligned}
    \end{equation*}  
\end{lemma}
One notices here that the difference between upper bounds for $\mathcal{I}_{p,d}(u;N)$ in Lemma $\ref{lem2.1}$ and  $\ref{lem2.2}$ is the presence of the extra function $\Psi_{u+1}(\beta)$ in $T_p(u;\epsilon).$ We emphasize that this function plays a pivotal role in the effectiveness of major arcs estimates discussed in section 4. We pause here and briefly explain this importance of $\Psi_{u+1}(\beta)$. In the expression 
($\ref{2.42.42.4}$) of $S_p(u)$, when $p$ is even, and considering the mean value of exponential sums as weighted counts of integer solutions via orthogonality, the exponential sum

\begin{equation}\label{2.132.13}
    \sum_{1\leq y\leq N}\sum_{j\in \mathbb{Z}}\hat{\varphi}\left(\frac{j}{N^{u}}\right)e((\alpha_{d-1}-dy\alpha_d)j) 
\end{equation}
provides crucial information. In fact, since the weight $\hat{\varphi}(j/N^u)$ is rapidly decreasing when $|j|\geq N^{u}$, one could regard the summation of $j$ as running over $|j|\ll N^{u+\epsilon}.$ This means that the magnitude of coefficient $dyj$ of $\alpha_d$ may be seen to be at most $O(N^{u+1+\epsilon})$. Therefore,  one infers via orthogonality that integration over $\alpha_d$ in ($\ref{2.42.42.4}$) can be interpreted as the number of integer solutions $1\leq n_1,\ldots,n_s \leq N$ satisfying
$$n_1^d+n_2^d+\cdots+n_s^d=O(N^{u+1+\epsilon}).$$

This information, however, may be difficult to  detect if we naively pursue pointwise estimates for the exponential sum 
($\ref{2.132.13}$) before integrating over $\boldsymbol{\alpha}.$ By introducing an additional cut-off function $\Psi_{u+1}(\beta)$ in the definition of $T_p(u;\epsilon)$, we preserve such information even after doing pointwise estimates for the exponential sum $(\ref{2.132.13})$. 



\bigskip

\begin{proof}
By Lemma $\ref{lem2.1},$ it suffices to show that 
\begin{equation*}
   S_p(u)\ll N^{u+1}T_p(u).
\end{equation*}
Considering the definition $(\ref{2.0})$ of $\Psi_A$, we infer that for $\delta, \eta\in \R$ with $\|\eta\|_{\R/\mathbb{Z}}\leq N^{-u+\epsilon/2},$ one has 
\begin{equation}\label{3.1}
  \Psi_u(\delta;N)\ll  \Psi_{u-\epsilon}(\delta+\eta;N),
\end{equation}
for sufficiently large $N$ in terms of $\epsilon$. Notice here that the implicit constant depends on $\sup_{x\in \R}\varphi(x).$

Meanwhile, whenever $\|\beta\|_{\R/\mathbb{Z}}\leq 2N^{-1-u}$, one has $\|dy\beta\|_{\R/\mathbb{Z}}\leq N^{-u+\epsilon/2}$ for $1\leq y\leq N$. Hence, we deduce by $(\ref{3.1})$ with $\delta=\alpha_{d-1}-dy\alpha_d$ and $\eta=-dy\beta$ that
\begin{equation}\label{3.2}
\begin{aligned}
\Psi_u(\alpha_{d-1}-dy\alpha_d)&\ll N^{u+1}\int_0^1 \Psi_{u}(\alpha_{d-1}-dy\alpha_d)\Psi_{u+1}(\beta)d\beta\\
&\ll N^{u+1}\int_0^1 \Psi_{u-\epsilon}(\alpha_{d-1}-dy\alpha_d-dy\beta)\Psi_{u+1}(\beta)d\beta.
\end{aligned}
\end{equation}

Then, it follows from the definition ($\ref{2.32.3}$) of $S(u;\gamma)$ together with the bound $(\ref{3.2})$ that one has
  \begin{equation*}
    \begin{aligned}
       &S_p(u)\ll N^{u+1}\widetilde{T}_p(u,\epsilon),
    \end{aligned}
    \end{equation*}
    where 
    \begin{equation}\label{07.10.217}
\widetilde{T}_p(u;\epsilon)=\int_0^1\int_{[0,1)^d}|f_d(\boldsymbol{\alpha};2N)|^p \Psi_{u+1}(\beta) \sum_{1\leq y\leq N}\Psi_{u-\epsilon}(\alpha_{d-1}-dy\alpha_d-dy\beta)d\boldsymbol{\alpha}d\beta.   
    \end{equation}
    Note that the integrand in  ($\ref{07.10.217}$) is a $1$-periodic function in $\alpha_d$. Then, by change of variable $\alpha_d\rightarrow \alpha_d-\beta$, this completes Lemma $\ref{lem2.2}$.
\end{proof}

\bigskip

For application of Hardy-Littlewood circle method in the following sections, we provide a representation for the mean value $T_p(u;\epsilon)$. 
By applying the Poisson summation formula to 
$\Psi_{u-\epsilon}(\alpha_{d-1}-dy\alpha_d)$, one finds that
\begin{equation}\label{2.162.16}
    \Psi_{u-\epsilon}(\alpha_{d-1}-dy\alpha_d)=N^{-u+\epsilon } \sum_{j\in \mathbb{Z}}\hat{\varphi}\left(\frac{j}{N^{u-\epsilon}}\right)e((\alpha_{d-1}-dy\alpha_d)j).
\end{equation}
Let us define
\begin{equation}\label{06.30.217}
    G(\alpha_{d-1},\alpha_d):=N^{-u-1+\epsilon }\sum_{1 \leq y \leq N}  \sum_{j\in \mathbb{Z}}\hat{\varphi}\left(\frac{j}{N^{u-\epsilon}}\right)e((\alpha_{d-1}-dy\alpha_d)j).
\end{equation}
Since $G(\alpha_{d-1},\alpha_{d})=N^{-1}\sum_{1\leq y\leq N}\Psi_{u-\epsilon}(\alpha_{d-1}-dy\alpha_d)$ and $ \Psi_{u-\epsilon}(\alpha_{d-1}-dy\alpha_d)\geq0,$ note that $G(\alpha_{d-1},\alpha_d) \geq 0$ for any real numbers $\alpha_{d-1}$ and $\alpha_d$. 
Note also by the definition of $\varphi$ that $\hat{\varphi}\left(\frac{j}{N^{u-\epsilon}}\right)\ll_{\epsilon,K} |j|^{-K}$ for all $K>0$, whenever $|j|\geq N^{u}$. Then, by the triangle inequality,  we see that
\begin{equation}\label{2.18}
    G(\alpha_{d-1},\alpha_d) \ll_{\epsilon,K} N^{-u-1+\epsilon}\sum_{|j|\leq N^{u}}\biggl|\sum_{1\leq y\leq N}e(dyj\alpha_d)\biggr|+N^{-K},
\end{equation}
for all $K>0.$
We emphasize that we shall use the upper bound in $(\ref{2.18})$ rather than $G(\alpha_{d-1},\alpha_{d})$ itself, when we apply for the circle method in the following sections. 

Then, on recalling the definition $(\ref{Tpdef})$ of $T_p(u;\epsilon)$, we obtain from $(\ref{2.162.16})$ and $(\ref{06.30.217})$ that
\begin{equation}\label{2.19}
    T_p(u;\epsilon)= N^{} \int_{[0,1)^d}\int_0^1|f_d(\alpha_d-\beta,\tilde{\boldsymbol{\alpha}};2N)|^p \Psi_{u+1}(\beta) G(\alpha_{d-1},\alpha_d)d\beta d\boldsymbol{\alpha}.
\end{equation}
Here and throughout, we write
\begin{equation}\label{2.22}
    U_p(u):=\int_{[0,1)^d}\int_0^1|f_d(\alpha_d-\beta,\tilde{\boldsymbol{\alpha}};2N)|^p \Psi_{u+1}(\beta) G(\alpha_{d-1},\alpha_d)d\beta d\boldsymbol{\alpha}.
\end{equation}
Then, it follows from $(\ref{2.19})$ together with Lemma $\ref{lem2.2}$ that  one has
\begin{equation}\label{2.23}
     \mathcal{I}_{p,d}(u;N)\ll N^{u+1}(\log N)^p U_p(u).
\end{equation}
This completes the section.

\begin{remark}\label{remark2.3}
    For notational convenience, we leave comments on the case $d=2,$ separately. Recall the definition ($\ref{2.5}$) of $f_d(\mathfrak{c};2N,\gamma)$. Let us replace the definition $(\ref{2.32.3})$, $(\ref{Tpdef})$, $(\ref{2.22})$ of $S_p(u), T_p(u;\epsilon)$ and $U_p(u)$ with the followings:
    \begin{equation*}
        \begin{aligned}
S_p(u)&:=\sup_{\gamma\in \R}\int_{[0,1)^d}|f_d(\boldsymbol{\alpha};2N,\gamma)|^p  \cdot \sum_{1\leq y\leq N}\Psi_u(\alpha_{d-1}-dy\alpha_d) d\boldsymbol{\alpha}\\
T_p(u;\epsilon)&:=\sup_{\gamma\in \R}\int_{[0,1)^d}\int_0^1|f_d(\alpha_d-\beta,\tilde{\boldsymbol{\alpha}};2N,\gamma)|^p \Psi_{u+1}(\beta) \sum_{1\leq y\leq N}\Psi_{u-\epsilon}(\alpha_{d-1}-dy\alpha_d)d\beta d\boldsymbol{\alpha}\\
 U_p(u):&=\sup_{\gamma\in \R}\int_{[0,1)^d}\int_0^1|f_d(\alpha_d-\beta,\tilde{\boldsymbol{\alpha}};2N,\gamma)|^p \Psi_{u+1}(\beta) G(\alpha_{d-1},\alpha_d)d\beta d\boldsymbol{\alpha}.
        \end{aligned}
    \end{equation*}
One deduces from $(\ref{2.7})$ that Lemma $\ref{lem2.1}$, with the definition of $S_p(u)$ right above, holds for $d=2$, and thus Lemma $\ref{lem2.2}$ and the inequality $(\ref{2.23})$ hold. We emphasize that if one naviely follows the proof with these definitions throughout this paper, one infers that the extra factor $\gamma$ does not affect the final conclusions, and thus all theorems for the case $d=2$ are still valid. We omit the details to save the space.
\end{remark}

\bigskip

\section{Hardy-Littlewood dissections}\label{sec3}
Define
\begin{equation}\label{3.13.1}
    f(u)= \left\{ \begin{aligned}
       & u\ \ \text{when}\ 0<u\leq 1\\
       & 1\ \ \text{when}\ 1\leq u\leq d-1.
    \end{aligned} \right.
\end{equation}
For a given parameter $0<u\leq d-1$, we define the major arcs $\mathfrak{M}_u$ by 
\begin{equation*}
    \mathfrak{M}_u:=\bigcup_{\substack{0\leq a\leq q\leq N^{f(u)}\\ (q,a)=1}}\mathfrak{M}_{u}(q,a),
\end{equation*}
where 
\begin{equation*}
    \mathfrak{M}(q,a)=\left\{\alpha\in [0,1):\ \left|\alpha-\frac{a}{q}\right|\leq \frac{N^{f(u)}}{qN^{1+u}}\right\}.
\end{equation*}
Define the minor arcs $\mathfrak{m}_{u}:=[0,1)\setminus \mathfrak{M}_{u}.$

By using these major and minor arcs dissections, we define 
\begin{equation}
\begin{aligned}
    U_p(\mathfrak{M}_u,u):= \int_{\mathfrak{M}_u\times[0,1)^{d-1}}\int_0^1|f_d(\alpha_d-\beta,\tilde{\boldsymbol{\alpha}};2N)|^p \Psi_{u+1}(\beta) G(\alpha_{d-1},\alpha_d)d\beta d\boldsymbol{\alpha}.
    \end{aligned}
\end{equation}
and
\begin{equation*}
\begin{aligned}
   U_p(\mathfrak{m}_u,u):= \int_{\mathfrak{m}_u\times[0,1)^{d-1}}\int_0^1|f_d(\alpha_d-\beta,\tilde{\boldsymbol{\alpha}};2N)|^p \Psi_{u+1}(\beta) G(\alpha_{d-1},\alpha_d)d\beta d\boldsymbol{\alpha}.
    \end{aligned}
\end{equation*}
Then, we have
\begin{equation}\label{3.3}
\begin{aligned}
U_p(u)=U_p(\mathfrak{M},u)+U_p(\mathfrak{m},u).
\end{aligned}
\end{equation}

\bigskip

In the following two sections, we prove two lemmas below, respectively.
\begin{lemma}\label{lem3.1}
Assume that Conjecture 1.5 is true. Then, for $d\geq 3,$ we have the following:

(i) For $0< u\leq 1$ and  $p\geq d(d+1)-2u$, one has
 \begin{equation*}
     U_p(\mathfrak{M}_u,u)\ll N^{p-d(d+1)/2-u-1+\epsilon}.
 \end{equation*}

 (ii) For $M\in  [1,d-2]\cap \mathbb{Z}$, $M< u\leq M+1$ and $p\geq d(d+1)-\frac{2(M+1)}{M+2-u}$, one has
    \begin{equation*}
       U_p(\mathfrak{M}_u,u)\ll N^{p-d(d+1)/2-u-1+\epsilon}.
    \end{equation*}
\end{lemma}

\begin{lemma}\label{lem3.2}
 For $d\geq 3,$ we have the following:

 (i) For $0< u\leq 1$ and  $p\geq d(d+1)-2u$, one has  
    \begin{equation*}
        U_p(\mathfrak{m},u)\ll N^{p-d(d+1)/2-u-1+\epsilon}.
    \end{equation*}

(ii) For $1<u\leq d-1$ and $p\geq d(d+1)-\frac{2d}{d+1-u}$, one has   
\begin{equation*}
        U_p(\mathfrak{m},u)\ll N^{p-d(d+1)/2-u-1+\epsilon}.
    \end{equation*}
\end{lemma}
We notice here that in Lemma $\ref{lem3.2}$ we do not need to assume that Conjecture $\ref{06.29.conj14}$ is true.

\begin{proof}[Proof of Theorem 1.6]
Assume that Conjecture $\ref{06.29.conj14}$ is true. Then, by Lemma $\ref{lem3.1}$ and $\ref{lem3.2}$ together with ($\ref{2.23}$) and $(\ref{3.3})$, we conclude that for $d\geq 3,$ one has the following:

 (i) For $0< u\leq 1$ and $p_0=d(d+1)-2u$, one has
\begin{equation*}
   \mathcal{I}_{p_0,d}(u;N)\ll N^{\epsilon}(N^{p_0-d(d+1)/2}+ N^{p_0/2-u}). 
\end{equation*}

(ii)  For $1<u\leq d-1$ and $p_0= d(d+1)-\frac{2d}{d+1-u}$, one has   
   $$\mathcal{I}_{p_0,d}(u;N)\ll  N^{p_0-d(d+1)/2+\epsilon}.$$

Therefore, by application of H\"older's inequality and by the trivial estimate that whenever $p\geq p_0$ one has
 \begin{equation*}
     \mathcal{I}_{p,d}(u;N)\ll N^{p-p_0} \mathcal{I}_{p_0,d}(u;N),
 \end{equation*}
 we complete the proof of Theorem $\ref{thm1.6}$ for $d\geq 3.$ 
 
 Recall Remark $\ref{remark2.3}$. We emphasize again that Lemma $\ref{lem3.1}$ and $\ref{lem3.2}$, with $U_p(\mathfrak{M},u)$ and $U_p(\mathfrak{m},u)$ replaced by 
 \begin{equation}
\begin{aligned}
    U_p(\mathfrak{M}_u,u):=\sup_{\gamma\in \R} \int_{\mathfrak{M}_u\times[0,1)^{d-1}}\int_0^1|f_d(\alpha_d-\beta,\tilde{\boldsymbol{\alpha}};2N,\gamma)|^p \Psi_{u+1}(\beta) G(\alpha_{d-1},\alpha_d)d\beta d\boldsymbol{\alpha}
    \end{aligned}
\end{equation}
and
\begin{equation*}
\begin{aligned}
   U_p(\mathfrak{m}_u,u):=\sup_{\gamma\in \R} \int_{\mathfrak{m}_u\times[0,1)^{d-1}}\int_0^1|f_d(\alpha_d-\beta,\tilde{\boldsymbol{\alpha}};2N,\gamma)|^p \Psi_{u+1}(\beta) G(\alpha_{d-1},\alpha_d)d\beta d\boldsymbol{\alpha},
    \end{aligned}
\end{equation*}
can be proved for $d=2$ by the same way in the proof of Lemma $\ref{lem3.1}$ and $\ref{lem3.2}$ in following sections. Therefore, by the same treatment for $d\geq 3$ above, we complete the proof of Theorem $\ref{thm1.6}$ for $d=2.$
\end{proof}

\section{Major arcs estimates}\label{sec4}
The purpose of this section is to prove Lemma $\ref{lem3.1}$. The main ingredient in this section is a variant of \cite[Lemma 2]{MR0913447}. We record here this variant of \cite[Lemma 2]{MR0913447} as a lemma. Combining this with refined shifting variables argument introduced in section $\ref{sec2}$ and interpolation theory may be useful in obtaining upper bounds for $L^p$-norm of exponential sums over the major arcs, for $p$ positive numbers.  The following lemma  is stated in greater generality than is strictly required here, to facilitate future applications. More precisely, we shall use the following lemma only for the case $1\leq \nu\leq 2.$
\begin{lemma}\label{lem4.1}
    Let $Z$ and $Y$ be positive numbers with $Y\leq Z.$ For $1\leq a\leq q\leq Y$ with $(q,a)=1$. Let $\mathfrak{M}(q,a)$ denote the interval $\left[\frac{a}{q}-\frac{Y}{qZ},\frac{a}{q}+\frac{Y}{qZ}\right]$ and assume that the  $\mathfrak{M}(q,a)$ are pairwise disjoint. Write $\mathfrak{M}$ for the union of $\mathfrak{M}(q,a)$ for all $1\leq a\leq q\leq Y$ with $(q,a)=1$. Let $G: \mathfrak{M}\rightarrow \mathbb{C}$ be a function satisfying 
    \begin{equation*}
        G(\alpha)\ll q^{-1}\left(1+Z\left|\alpha-\frac{a}{q}\right|\right)^{-1}\ \ \textrm{for}\ \ \alpha\in \mathfrak{M}(q,a).
    \end{equation*}
    Furthermore, let $H$ be a positive number with $\log H\asymp \log Z.$ 
    Let $F:\R\rightarrow [0,\infty)$ be a function with a Fourier expansion
    \begin{equation*}
        F(\alpha):=\sum_{h\in \mathbb{Z}}\psi_he(\alpha h),
    \end{equation*}
    where $|\psi_h|\ll_K |h|^{-K}$ for all $K>0$ whenever $|h|> H$. Then, for $0\leq\nu\leq 2,$ one has
    \begin{equation*}
\begin{aligned}  &\int_{\mathfrak{M}}F(\alpha)|G(\alpha)|^{\nu}d\alpha\\
&\ll_K \left\{ \begin{aligned}
       & Z^{\epsilon-1}Y^{2-\nu}|\psi_0|+Z^{\epsilon-1}Y^{1-\nu}\sum_{|h|\leq H}|\psi_h|+H^{-K},\ \ \text{when}\ 0\leq\nu\leq 1\\
       & Z^{\epsilon-1}Y^{2-\nu}|\psi_0|+Z^{\epsilon-1}\sum_{|h|\leq H}|\psi_h|+H^{-K},\ \ \  \ \ \ \ \ \text{when}\ 1\leq \nu\leq 2,
    \end{aligned} \right.
\end{aligned}
    \end{equation*}
for all $K>0.$
\end{lemma}

\begin{proof}
On writing $\beta=\alpha-a/q$, we have
\begin{equation}\label{4.1}
\begin{aligned}
    \int_{\mathfrak{M}}F(\alpha)|G(\alpha)|^{\nu}d\alpha&\ll \sum_{1\leq q\leq Y}\sum_{\substack{1\leq a\leq q\\(q,a)=1}}\int_{|\beta|\leq Y/(qZ)}q^{-\nu}(1+Z|\beta|)^{-\nu}F(\beta+a/q)d\beta\\
    &\ll \sum_{1\leq q\leq Y}\sum_{\substack{1\leq a\leq q\\(q,a)=1}}\sum_{h\in \mathbb{Z}}q^{-\nu}\psi_h\rho_h e\left(\frac{ah}{q}\right),
\end{aligned}
\end{equation}
where $$\rho_h=\int_{|\beta|\leq Y/(qZ)}(1+Z|\beta|)^{-\nu}e(\beta h)d\beta.$$
Notice that one has
\begin{equation}\label{4.24.2}
    \rho_h\ll \left\{ \begin{aligned}
       & Z^{-1}Y^{1-\nu} q^{\nu-1}(\log Z),\ \ \text{when}\ 0\leq\nu\leq 1.\\
       & Z^{-1}(\log Z),\ \ \ \ \ \ \  \ \ \ \ \ \ \ \ \text{when}\ 1\leq \nu\leq 2.
    \end{aligned} \right. 
\end{equation}
Furthermore, it follows by \cite[Theorem 271]{MR0568909} that for $h\neq 0$, one has
\begin{equation}\label{4.2}
    \sum_{\substack{1\leq a\leq q\\(q,a)=1}}e\left(\frac{ah}{q}\right)\leq \sum_{d|(q,h)}d.
\end{equation}
Therefore, when $0\leq\nu\leq 1,$ we deduce from $(\ref{4.1})$, $(\ref{4.24.2})$ and $(\ref{4.2})$ that
\begin{equation}\label{4.3}
\begin{aligned}
     &\int_{\mathfrak{M}}F(\alpha)|G(\alpha)|^{\nu}d\alpha\\
     &\ll Z^{-1}Y^{1-\nu}\log Z\biggl(|\psi_0|\sum_{1\leq q\leq Y}1+\sum_{h\neq 0}|\psi_h|\sum_{1\leq q\leq Y}\sum_{d|(q,h)}(d/q)\biggr)\\
     &\ll Z^{-1}Y^{1-\nu}\log Z\biggl(Y|\psi_0|+\sum_{h\neq 0}|\psi_h|\sum_{d|h}\sum_{r\leq Y/d}1/r\biggr).
\end{aligned}
\end{equation}
By the property of $\psi_h$, the second term in the last expression above is  
\begin{equation}\label{4.4}
  \ll \sum_{|h|\leq H}|\psi_h|\sum_{d|h}\sum_{r\leq Y/d}1/r+H^{-K}\ll  (\log Z)\sum_{|h|\leq H}|\psi_h|+H^{-K},
\end{equation}
for all $K>0.$
On substituting $(\ref{4.4})$ into $(\ref{4.3})$, whenever $0\leq\nu\leq 1,$ we conclude that
\begin{equation*}
       \begin{aligned}
\int_{\mathfrak{M}}F(\alpha)|G(\alpha)|^{\nu}d\alpha\ll_K Z^{\epsilon-1}Y^{1-\nu}\biggl(Y|\psi_0|+\sum_{h\neq 0}|\psi_h|\biggr)+H^{-K},
       \end{aligned}
\end{equation*}
for all $K>0.$
Next, when $1\leq \nu\leq 2,$ we obtain from $(\ref{4.1})$, $(\ref{4.24.2})$ and $(\ref{4.2})$ that
\begin{equation}\label{4.6}
    \begin{aligned}
        &\int_{\mathfrak{M}}F(\alpha)|G(\alpha)|^{\nu}d\alpha\\
        &\ll Z^{-1}Y^{2-\nu}\log Z|\psi_0|+Z^{-1}\sum_{h\neq 0}|\psi_h|\sum_{1\leq q\leq Y}q^{1-\nu}\sum_{d|(q,h)}(d/q)\\
        &\ll Z^{-1}Y^{2-\nu}\log Z|\psi_0|+Z^{-1}\sum_{h\neq 0}|\psi_h|\sum_{d|h}\sum_{r\leq Y/d}1/r
    \end{aligned}
\end{equation}
By the property of $\psi_h$ again, the second term in the last expression above is 
\begin{equation}\label{4.7}
\begin{aligned}
  \ll Z^{-1}\sum_{|h|\leq H}|\psi_h|\sum_{d|h}\sum_{r\leq Y/d}1/r+H^{-K}\ll Z^{-1}\log Z\sum_{|h|\leq H}|\psi_h|+H^{-K},
\end{aligned}
\end{equation}
for all $K>0.$
By substituting $(\ref{4.7})$ into $(\ref{4.6})$, whenever $1\leq \nu\leq 2$, we conclude that
\begin{equation*}
    \int_{\mathfrak{M}}F(\alpha)|G(\alpha)|^{\nu}d\alpha\ll_K Z^{\epsilon-1}Y^{2-\nu}|\psi_0|+Z^{\epsilon-1}\sum_{|h|\leq H}|\psi_h|+H^{-K},
\end{equation*}
for all $K>0.$
\end{proof}

\begin{proof}[Proof of Lemma 3.1]

We first note that Conjecture \ref{06.29.conj14} implies that \eqref{06.29.11} is true for $d-1 < u \leq d$ and even $p$. To see this, observe that the first term on the right hand side of \eqref{06.29.11} is independent of $u$ and this term dominates the second term on the right hand side of \eqref{06.29.11} when $p \geq d(d+1)-2u$. So for $d-1 < u \leq d$, when $p \geq d(d+1)-2(d-1)$ is even, by simply enlarging the domain, we can obtain \eqref{06.29.11}. Since we are interested in an even number $p$, the remaining case is that $p \leq d(d+1)-2d$. For this range, \eqref{06.29.11} follows simply by applying the main conjecture of Vinogradov's mean value theorem proved by \cite{MR3548534, MR3479572, MR3938716} on $\alpha_1,\ldots,\alpha_{d-1}$ variables. This gives a proof that \eqref{06.29.11} is true for $d-1 < u \leq d$ and even $p$.

  We shall prove that for $0< u\leq 1$ and $p_0=d(d+1)-2u$, one has
   \begin{equation}\label{4.74.7}
U_{p_0}(\mathfrak{M}_u,u)\ll N^{p_0-d(d+1)/2-u-1+\epsilon},
   \end{equation}
  and prove that for $M\in  [1,d-2]\cap \mathbb{Z}$, $M< u\leq M+1$ and $p_0= d(d+1)-\frac{2(M+1)}{M+2-u}$, one has
   \begin{equation}\label{4.94.9}
U_{p_0}(\mathfrak{M}_u,u)\ll N^{p_0-d(d+1)/2-u-1+\epsilon}.
   \end{equation}
   
Define the mean value $U_p^{(\nu)}(\mathfrak{M}_u,u)$ by
\begin{equation}\label{4.9}
\begin{aligned}
  &U_p^{(\nu)}(\mathfrak{M}_u,u):=\int_{\mathfrak{M}_u\times[0,1)^{d-1}}\int_0^1|f_d(\alpha_d-\beta,\tilde{\boldsymbol{\alpha}};2N)|^p \Psi_{u+1}(\beta) |G(\alpha_{d-1},\alpha_{d})|^{\nu}d\beta d\boldsymbol{\alpha}.
\end{aligned}
\end{equation}
Recall the inequality \eqref{2.18}, that is
\begin{equation}
    G(\alpha_{d-1},\alpha_d) \ll_{\epsilon,K} N^{-u-1+\epsilon}\sum_{|j|\leq N^{u}}\biggl|\sum_{1\leq y\leq N}e(dyj\alpha_d)\biggr| +N^{-K},
\end{equation}
for all $K>0.$
Notice that the right hand side is independent of the variable $\alpha_{d-1}$.
By abusing the notation, let us denote by $G(\alpha_d)$ the first term in the right hand side. Thus, we find that
\begin{equation}\label{upperbound}
    U_p^{(\nu)}(\mathfrak{M}_u,u)\ll \int_{\mathfrak{M}_u\times[0,1)^{d-1}}\int_0^1|f_d(\alpha_d-\beta,\tilde{\boldsymbol{\alpha}};2N)|^p \Psi_{u+1}(\beta) |G(\alpha_d)|^{\nu}d\beta d\boldsymbol{\alpha}+N^{-K},
\end{equation}
for all $K>0.$

We provide an upper bound for $U_{2s}^{(\nu)}(\mathfrak{M}_u,u)$, especially when $s\in \mathbb{N}$. For $s\in \mathbb{N}$ and $1\leq i\leq d,$ we write 
$$\sigma_{s,i}(\boldsymbol{n})=n_1^{i}+\cdots+n_s^i-n_{s+1}^{i}-\cdots-n_{2s}^i.$$
 We find by expanding $2s$-th powers that 
\begin{equation}
|f_d(\alpha_d-\beta,\tilde{\boldsymbol{\alpha}};2N)|^{2s}=\sum_{1\leq \boldsymbol{n}\leq 2N}e\biggl((\alpha_d-\beta)\sigma_{s,d}(\boldsymbol{n})+\sum_{i=1}^{d-1}\alpha_i\sigma_{s,i}(\boldsymbol{n})\biggr).
\end{equation}
By rearranging the summands in the sum in terms of the value of $\sigma_{s,d}(\boldsymbol{n})$, we see that
\begin{equation}\label{06.27.416}
    |f_d(\alpha_d-\beta,\tilde{\boldsymbol{\alpha}};2N)|^{2s}=\sum_{|h|\leq 2^{d+1}sN^d}\sum_{\substack{1\leq \boldsymbol{n}\leq 2N\\ \sigma_{s,d}(\boldsymbol{n})=h}}e\biggl((\alpha_d-\beta)h+\sum_{i=1}^{d-1}\alpha_i\sigma_{s,i}(\boldsymbol{n})\biggr).
\end{equation}
Hence, we deduce that
\begin{equation}\label{4.11}
    \int_{[0,1)^{d-1}} |f_d(\alpha_d-\beta,\tilde{\boldsymbol{\alpha}};2N)|^{2s}d\tilde{\boldsymbol{\alpha}}=\sum_{|h|\leq 2^{d+1}sN^d}\N(h)e((\alpha_d-\beta)h),
\end{equation}
where
\begin{equation*}
    \N(h):=\{\boldsymbol{n}\in [1,2N]^{2s}\cap \mathbb{Z}^{2s}|\ \sigma_{s,d}(\boldsymbol{n})=h,\ \sigma_{i,d}(\boldsymbol{n})=0\ (1\leq i\leq d-1)\}.
\end{equation*}
Therefore, by substituting $(\ref{4.11})$ into the upper bound in $(\ref{upperbound})$ of $U_{p}^{(\nu)}(\mathfrak{M}_u,u)$ with $p=2s,$ that 
\begin{equation}\label{4.12}
  U_{2s}^{(\nu)}(\mathfrak{M}_u,u)\ll  \int_{\mathfrak{M}_u}\int_0^1 \sum_{|h|\leq 2^{d+1}sN^d}\N(h)e((\alpha_d-\beta)h)\Psi_{u+1}(\beta)|G(\alpha_d)|^{\nu}d\beta d\alpha_d+N^{-K},
\end{equation}
for all $K>0.$
Meanwhile, by applying the Poisson summation formula, one has
\begin{equation}\label{4.13}
    \Psi_{u+1}(\beta)=N^{-u-1}\sum_{j\in \mathbb{Z}}\hat{\varphi}\left(\frac{j}{N^{u+1}}\right)e(\beta j)
\end{equation}
By substituting this into $(\ref{4.12})$ and by taking integration over $\beta$, we find by orthogonality that
\begin{equation}\label{4.154.15}
\begin{aligned}
    &U_{2s}^{(\nu)}(\mathfrak{M}_u,u)\\&\ll N^{-u-1}\int_{\mathfrak{M}_u} \biggl(\sum_{|h|\leq 2^{d+1}sN^d}\hat{\varphi}\left(\frac{h}{N^{u+1}}\right)\N(h)e(\alpha_dh)\biggr)|G(\alpha_d)|^{\nu}d\alpha_d+N^{-K},
\end{aligned}
\end{equation}
for all $K>0.$
We shall apply Lemma $\ref{lem4.1}$ to the upper bound in $(\ref{4.154.15})$ of $U_{2s}^{(\nu)}(\mathfrak{M}_u,u)$. To do this, we first see that 
 \begin{equation*}
 \begin{aligned}
     G(\alpha_d)\ll N^{-u-1+\epsilon}\sum_{|j|\leq N^u}\text{min}\left\{N,\ \frac{1}{\|dj\alpha_d\|_{\R/\mathbb{Z}}}\right\},
 \end{aligned}
 \end{equation*}
 and thus by \cite[Lemma 3.2]{MR865981} whenever $|\alpha_d-a/q|\leq q^{-2}$ with $(q,a)=1$ one has
 \begin{equation}\label{07.05.417}
 \begin{aligned}
     G(\alpha_d)\ll N^{\epsilon}\left(\frac{1}{q}+\frac{1}{N^u}+\frac{1}{N}+\frac{q}{N^{u+1}}\right).
 \end{aligned}
 \end{equation}
 By applying a standard transference principle \cite[Lemma A.1]{MR3335281} with  $\theta=1$, $X=N^{\epsilon}$ and $Y=\min\{N^{u},N\}$ and $Z=N^{u+1}$, one deduces that whenever $b\in \mathbb{Z}$ and $r\in \mathbb{N}$ satisfying $(b,r)=1$ and $|\alpha_d-b/r|\leq r^{-2},$ one has
 \begin{equation*}
     G(\alpha_d)\ll N^{\epsilon}(\lambda^{-1}+N^{-u}+N^{-1}+\lambda N^{-u-1}),
 \end{equation*}
where $\lambda=r+N^{u+1}|r\alpha_d-b|.$ When $\alpha_d\in \mathfrak{M}_u(r,b)\subseteq \mathfrak{M}_u$, one has  $r\leq N^{f(u)}$ and $N^{u+1}|r\alpha_d-b|\leq N^{f(u)}$, so that $\lambda\leq 2N^{f(u)}$. We therefore see that under such circumstances, one has
\begin{equation}\label{4.15}
    G(\alpha_d)\ll N^{\epsilon}\Phi(\alpha_d),
\end{equation}
where $\Phi(\alpha_d)$ is the function taking the value $(q+N^{u+1}|q\alpha_d-a|)^{-1},$ when one has $\alpha_d\in \mathfrak{M}_u(q,a)\subseteq \mathfrak{M}_u,$ and otherwise $\Phi(\alpha_d)=0.$
Next, for any $\epsilon>0,$ it follows from the decay rate of $\hat{\varphi}$ that whenever $|h|\geq N^{u+1+\epsilon}$, one has 
\begin{equation}\label{4.184.18}
    \hat{\varphi}\left(\frac{h}{N^{u+1}}\right)\N(h)\ll_{\epsilon,K} |h|^{-K}\ \ \text{for all}\ K>0,
\end{equation}
since $\N(h)\ll N^{p}$ obviously.

\bigskip

We shall prove $(\ref{4.74.7})$ and $(\ref{4.94.9}),$ separately.

\bigskip

(a) The case that $0< u\leq 1$ and $p_0=d(d+1)-2u$.

\bigskip

Consider $2s \leq d(d+1)$.
By $(\ref{4.15})$ and $(\ref{4.184.18})$, and by applying Lemma $\ref{lem4.1}$ to the upper bound in $(\ref{4.154.15})$ of $U_{2s}^{(1)}(\mathfrak{M}_u,u)$, with 
\begin{equation*}
 Z=N^{u+1},\ Y=N^u,\ H=N^{u+1+\epsilon},\ \mathfrak{M}(q,a)=\mathfrak{M}_u(q,a),\   \psi_h=\hat{\varphi}\left(\frac{h}{N^{u+1}}\right)\N(h),
\end{equation*}
 one has
\begin{equation}\label{4.17}
    U_{2s}^{(1)}(\mathfrak{M}_u,u)\ll N^{-(u+1)}\biggl(N^{\epsilon-(u+1)}N^{u}N^{s}+N^{\epsilon-(u+1)}\sum_{|h|\leq H}|\psi_h|\biggr)+N^{-K},
\end{equation}
for all $K>0.$
By the definition of $\N(h)$ and $\hat{\varphi}(x)\ll 1$ for $x\in \R$, we see that $\sum_{|h|\leq H}|\psi_h|$ is 
\begin{equation}\label{4.1919}
    \ll \# \{\boldsymbol{n}\in [1,N]^{2s}\cap \mathbb{Z}^{2s}:\ |\sigma_{s,d}(\boldsymbol{n})|\leq H,\  \sigma_{s,i}(\boldsymbol{n})=0\ (1\leq i\leq d-1)\}.
\end{equation}
Under the assumption that Conjecture $\ref{06.29.conj14}$ is true, one infers by \cite[Lemma 2.1]{MR1002452} with $x_i=\alpha_i$, $w_k(u)=u^k$, $K=d$, $M=s,$ $\mathscr{A}=[1,N]$, $\delta_d=H$, $\delta_k=1\ (1\leq k\leq d-1)$ that
\begin{equation}\label{4.18}
\begin{aligned}
    &\# \{\boldsymbol{n}\in [1,N]^{2s}\cap \mathbb{Z}^{2s}:\ |\sigma_{s,d}(\boldsymbol{n})|\leq H,\  \sigma_{s,i}(\boldsymbol{n})=0\ (1\leq i\leq d-1)\}\\
   & \ll N^{\epsilon}(N^{s}+N^{2s-d(d+1)/2+(u+1)}).
\end{aligned}
\end{equation}
Consequently, it follows by $(\ref{4.17}), (\ref{4.1919})$ and $(\ref{4.18})$ that whenever $2s\geq d(d+1)-2$, one has
\begin{equation}\label{4.25}
    U_{2s}^{(1)}(\mathfrak{M}_u,u)\ll N^{s-2- u+\epsilon}+N^{2s-d(d+1)/2-u-1+\epsilon}\ll N^{2s-d(d+1)/2-u-1+\epsilon}.
\end{equation}
Note that  $p_0=d(d+1)-2u$ is between $d(d+1)-2$ and $d(d+1)$, for $0\leq u\leq 1$. Then, by applying interpolation between bounds $(\ref{4.25})$ with $2s=d(d+1)-2$ and $2s=d(d+1)$, one has
\begin{equation}
      U_{p_0}^{(1)}(\mathfrak{M}_u,u)\ll  N^{p_0-d(d+1)/2-u-1+\epsilon}.
\end{equation}
This confirms $(\ref{4.74.7}).$


\bigskip

(b) The case that $M\in  [1,d-2]\cap \mathbb{Z}$, $M< u\leq M+1$ and $p_0= d(d+1)-\frac{2(M+1)}{M+2-u}$.

\bigskip

By applying the H\"{o}lder's inequality, we have
\begin{equation}\label{06.27.52}
    U_{p_0}^{(1)}(\mathfrak{M}_u,u) \ll (U_{d(d+1)}^{(0)}(\mathfrak{M}_u,u))^{\frac1q}
    (U_{d(d+1)-2(M+1)}^{(q')}(\mathfrak{M}_u,u))^{\frac{1}{q'}},
\end{equation}
where
\begin{equation}
\begin{split}
     q=\frac{2(M+1)}{2(M+1)-d(d+1)+p_0}\ \text{and} \;\;\; q'=\frac{2(M+1)}{d(d+1)-p_0}.
\end{split}
\end{equation}
Note by the main conjecture of Vinogradov's mean value theorem proved by \cite{MR3548534, MR3479572, MR3938716} that
\begin{equation}\label{06.27.54}
\begin{aligned}
    U_{d(d+1)}^{(0)}(\mathfrak{M}_u,u)&:=\int_{\mathfrak{M}_u\times[0,1)^{d-1}}\int_0^1|f_d(\alpha_d-\beta,\tilde{\boldsymbol{\alpha}};2N)|^{d(d+1)} \Psi_{u+1}(\beta) d\beta d\boldsymbol{\alpha}\\
    &\leq \int_0^1 \biggl(\int_{[0,1)^d}|f_d(\alpha_d-\beta,\tilde{\boldsymbol{\alpha}};2N)|^{d(d+1)} d\boldsymbol{\alpha}\biggr) \Psi_{u+1}(\beta)  d\beta \\
    &\ll N^{\frac{d(d+1)}{2}+\epsilon}\int_0^1 \Psi_{u+1}(\beta)  d\beta \ll N^{\frac{d(d+1)}{2}-u-1+\epsilon}.
\end{aligned}
\end{equation}

Consider $2s \leq d(d+1)$. By $(\ref{4.15})$ and $(\ref{4.184.18})$, and by applying Lemma $\ref{lem4.1}$ again to the upper bound in $(\ref{4.154.15})$ of $U_{2s}^{(\nu)}(\mathfrak{M}_u,u;\gamma)$, with 
\begin{equation*}
  Z=N^{u+1},\ Y=N,\ H=N^{u+1+\epsilon},\ \mathfrak{M}(q,a)=\mathfrak{M}_u(q,a),\   \psi_h=\hat{\varphi}\left(\frac{h}{N^{u+1}}\right)\N(h),
\end{equation*}
one has
\begin{equation*}
      U_{2s}^{(\nu)}(\mathfrak{M}_u,u)\ll N^{-(u+1)}\biggl(N^{\epsilon-(u+1)}N^{2-\nu}N^{s}+N^{\epsilon-(u+1)}\sum_{|h|\leq H}|\psi_h|\biggr)+N^{-K},
\end{equation*}
for all $K>0.$
By the same treatment leading to $(\ref{4.18})$, whenever $1\leq \nu\leq 2,$ we see that
\begin{equation}\label{4.2424}
     U_{2s}^{(\nu)}(\mathfrak{M}_u,u)\ll N^{s-2u-\nu+\epsilon}+N^{2s-\frac{d(d+1)}{2}-u-1+\epsilon}.
\end{equation}
By using this bound, we find that 
\begin{equation}\label{4.2727}
    U_{d(d+1)-2(M+1)}^{(q')}(\mathfrak{M}_u,u)\ll N^{-2u-q'+\frac{d(d+1)}{2}-(M+1)+\epsilon }+N^{-u+\frac{d(d+1)}{2}-2M-3+\epsilon }.
\end{equation}
For $p_0= d(d+1)-\frac{2(M+1)}{M+2-u},$ we see that the first and second terms in the right hand side of $(\ref{4.2727})$ give the same bound. Hence, by substituting $(\ref{06.27.54})$ and $(\ref{4.2727})$ into $(\ref{06.27.52})$, we find that
\begin{equation*}
\begin{aligned}
     U_{p_0}^{(1)}(\mathfrak{M}_u,u)&\ll N^{\epsilon}\cdot N^{(d(d+1)/2-u-1)/q}\cdot N^{(-u+d(d+1)/2-2M-3)/q'}\\
     &\ll N^{\epsilon}\cdot N^{(d(d+1)/2-u-1)(1/q+1/q')}\cdot N^{(-2M-2)/q'}\\
     &\ll N^{\epsilon}\cdot N^{(d(d+1)/2-u-1)}\cdot N^{p_0-d(d+1)}\\
     &\ll N^{p_0-d(d+1)/2-u-1+\epsilon},
\end{aligned}
\end{equation*}
which confirms $(\ref{4.94.9}).$ We have confirmed $(\ref{4.74.7})$ and $(\ref{4.94.9}).$ Then, on observing by the trivial estimate that for $p\geq p_0$ one has
$$ U_{p}^{(1)}(\mathfrak{M}_u,u)\ll N^{p-p_0}U_{p_0}^{(1)}(\mathfrak{M}_u,u),$$
and we complete the proof of Lemma $\ref{lem3.1}$.
\end{proof}

\section{Minor arcs estimates}\label{sec5}
In this section, we prove Lemma $\ref{lem3.2}.$ 
Recall the definition ($\ref{06.30.217}$) of $G(\alpha_d,\alpha_{d-1}),$ and by ($\ref{2.18}$), one has
\begin{equation}
    G(\alpha_{d-1},\alpha_d) \ll_{\epsilon,K} N^{-u-1+\epsilon}\sum_{|j|\leq N^{u}}\biggl|\sum_{1\leq y\leq N}e(dyj\alpha_d)\biggr| +N^{-K}
\end{equation}
for all $K>0$. Note again by \eqref{07.05.417} that
 whenever $|\alpha_d-a/q|\leq 1/q^2$, one has
\begin{equation*}
    G(\alpha_{d-1},\alpha_d)\ll N^{\epsilon}\left(\frac{1}{q}+\frac{1}{N^u}+\frac{1}{N}+\frac{q}{N^{1+u}}\right) + N^{-K}.
\end{equation*}
 Recall the definition ($\ref{3.13.1}$) of $f(u)$. Then, for $\alpha_d\in \mathfrak{m}_u$ and $\alpha_{d-1} \in [0,1)$, one has  
\begin{equation}\label{0624.52}
    G(\alpha_{d-1},\alpha_d)\ll N^{-f(u)+\epsilon}.
\end{equation}
We will use this bound.
\\

Let us now prove Lemma \ref{lem3.2}. Suppose that $p \leq d(d+1)$.
We start with
\begin{equation}\label{0702.52}
\begin{aligned} 
U_p(\mathfrak{m}_u,u) \ll \int_{[0,1)^d}\int_0^1|f_d(\alpha_d-\beta,\tilde{\boldsymbol{\alpha}};2N)|^p \Psi_{u+1}(\beta)d\beta d\boldsymbol{\alpha}\sup_{\boldsymbol{\alpha} \in \mathfrak{m}_u \times [0,1) }|G(\alpha_{d-1},\alpha_d)|.
    \end{aligned}
\end{equation}
By the change of variables $\alpha_d \mapsto \alpha+\beta$ and evaluating the integral with respect to $\beta$, the right hand side is bounded by
\begin{equation}\label{0702.53}
    N^{-u-1}\int_{[0,1)^d}|f_d(\alpha_d,\tilde{\boldsymbol{\alpha}};2N)|^p  d\boldsymbol{\alpha}\sup_{\boldsymbol{\alpha} \in \mathfrak{m}_u \times [0,1) }|G(\alpha_{d-1},\alpha_d)|.
\end{equation}
By the main conjecture of Vinogradov's mean value theorem proved by \cite{MR3548534, MR3479572, MR3938716} and \eqref{0624.52}, this is bounded by
\begin{equation}
    N^{-u-1}N^{\frac{p}{2}}N^{-f(u)}N^{\epsilon}.
\end{equation}
So we have obtained
\begin{equation}\label{07.02.55}
    U_p(\mathfrak{m}_u,u) \ll N^{-u-1}N^{\frac{p}{2}}N^{-f(u)}N^{\epsilon}.
\end{equation}
Let us consider the case $0<u \leq 1$. Since $f(u)=u$, what we need to check is
\begin{equation}
    N^{-u-1}N^{\frac{p}{2}}N^{-u} \ll N^{p-d(d+1)/2-u-1}   
\end{equation}
for $p \geq d(d+1)-2u$. By direct computations, one can see that it is true.\\ 

Let us next consider the case $1 <u \leq d-1$. By H\"{o}lder's inequality,
\begin{equation}\label{07.02.56}
    U_p(\mathfrak{m}_u,u) \ll U_{d(d-1)}(\mathfrak{m}_u,u)^{\frac{p}{d(d-1)}\theta} U_{d(d+1)}(\mathfrak{m}_u,u)^{\frac{p}{d(d+1)}(1-\theta)},
\end{equation}
where
\begin{equation}
    \theta:=\frac{(d-1)(d^2+d-p)}{2p}.
\end{equation}
By \eqref{0624.52} and \eqref{07.02.55}, we have
\begin{equation}
    U_{d(d+1)}(\mathfrak{m}_u,u) \ll N^{\frac{d(d+1)}{2}-u-1-1+\epsilon}.
\end{equation}
We claim that
\begin{equation}\label{07.02.59}
    U_{d(d-1)}(\mathfrak{m}_u,u) \ll N^{-u-1}N^{-u}N^{\frac{d(d-1)}{2}}N^{\epsilon}.
\end{equation}
Let us assume this claim and finish the proof. By \eqref{07.02.56}, it suffices to check that
\begin{equation}
    (N^{\frac{d(d-1)}{2}-2u-1})^{ \frac{p\theta}{d(d-1)}}(N^{\frac{d(d+1)}{2}-u-2})^{\frac{(1-\theta)p}{d(d+1)}} \ll N^{p-\frac{d(d+1)}{2}-u-1 }
\end{equation} 
for $p\geq d(d+1)-\frac{2d}{d+1-u}$. Direct computations show that it is the case.
\\

Let us now prove the claim. We start with
\begin{equation*}
\begin{aligned}
   U_{d(d-1)}(\mathfrak{m}_u,u) \ll \int_{[0,1)^d}\int_0^1|f_d(\alpha_d-\beta,\tilde{\boldsymbol{\alpha}};2N)|^{d(d-1)} \Psi_{u+1}(\beta) G(\alpha_{d-1},\alpha_d)d\beta d\boldsymbol{\alpha}.
    \end{aligned}
\end{equation*}
By \eqref{2.162.16}-\eqref{06.30.217}, the right hand side is equal to
\begin{equation*}
\frac{1}{N}\sum_{1 \leq y \leq N}
\int_{[0,1)^d}\int_0^1|f_d(\alpha_d-\beta,\tilde{\boldsymbol{\alpha}};2N)|^{d(d-1)} \Psi_{u+1}(\beta) 
    \Psi_{u-\epsilon}(\alpha_{d-1}-dy\alpha_d)d\beta d\boldsymbol{\alpha}.
\end{equation*}
After the change of variables $\alpha_d \mapsto \alpha_d+\beta$, the above expression becomes
\begin{equation*}
    \frac1N\sum_{y}\int_{[0,1)^d}\int_0^1|f_d(\alpha_d,\tilde{\boldsymbol{\alpha}};2N)|^{d(d-1)} \Psi_{u+1}(\beta) 
    \Psi_{u-\epsilon}(\alpha_{d-1}-dy\alpha_d-dy\beta )d\beta d\boldsymbol{\alpha}.
\end{equation*}
Thus, in order to prove the claim, it suffices to prove
\begin{equation}\label{07.02.511}
\int_{[0,1)^d}|f_d(\alpha_d,\tilde{\boldsymbol{\alpha}};2N)|^{d(d-1)}  
    \Psi_{u-\epsilon}(\alpha_{d-1}-dy\alpha_d-dy\beta ) d\boldsymbol{\alpha} \ll N^{\frac{d(d-1)}{2}-u+\epsilon}
\end{equation}
for every $|\beta| \leq N^{-u-1+\epsilon}$ and $1 \leq y \leq N$. After the change of variables $\alpha_{d-1} \mapsto \alpha_{d-1}+dy\alpha_d+dy\beta$, the left hand side of \eqref{07.02.511} becomes
\begin{equation}\label{07.03.513}
    \int_{[0,1)^d}\Big| \sum_{1 \leq n \leq N}a_n e( \alpha_d (n^d+dyn^{d-1})+ \alpha_{d-1}n^{d-1}+\cdots+\alpha_1n ) \Big|^{d(d-1)} \Psi_{u-\epsilon}(\alpha_{d-1}) d\boldsymbol{\alpha},
\end{equation}
where $a_n:=e(dy \beta n^{d-1})$. We do change of variables $n \mapsto n-y$. More precisely, we write 
\begin{equation}
\begin{split}
    &\sum_{1 \leq n \leq N}a_n e( \alpha_d (n^d+dyn^{d-1})+ \alpha_{d-1}n^{d-1}+\cdots+\alpha_1n )
    \\&
    =\sum_{1+y \leq n \leq N+y}a_{n-y} e( \beta_d n^d+ \beta_{d-1}n^{d-1}+\cdots+\beta_1n )
\end{split}
\end{equation}
for some $\beta_1,\ldots,\beta_n$ depending on $\alpha_i$'s. Then we do change of variables on $\alpha_i$'s  so that \eqref{07.03.513} becomes
\begin{equation*}
    \int_{[0,1)^d}\Big| \sum_{1+y \leq n \leq N+y}a_{n-y} e( \alpha_d n^d+ \alpha_{d-1}n^{d-1}+\cdots+\alpha_1n ) \Big|^{d(d-1)} \Psi_{u-\epsilon}(\alpha_{d-1}) d\boldsymbol{\alpha}.
\end{equation*}
Let us freeze $\alpha_{d-1}$ and define $b_n:=a_{n-y} e(\alpha_{d-1}n^{d-1})$. Since $|b_n|=1$, by an application of the main conjecture of Vinogradov's mean value theorem proved by \cite{MR3548534, MR3479572, MR3938716}, we have
\begin{equation*}
\begin{aligned}
    &\int_{[0,1)^{d-1}}\Big| \sum_{1+y \leq n \leq N+y}b_{n} e( \alpha_d n^d+  \alpha_{d-2}n^{d-2} +\cdots+\alpha_1n ) \Big|^{d(d-1)} d\alpha_d d\alpha_{d-2} \cdots d\alpha_1 \\&\ll N^{\frac{d(d-1)}{2}+\epsilon}.
\end{aligned}
\end{equation*}
After taking the integration with respect to $\alpha_{d-1}$ with the weight  $\Psi_{u-\epsilon}(\alpha_{d-1})$ on the both sides, we can see that the left hand side of \eqref{07.02.511} is bounded by $N^{\frac{d(d-1)}{2}-u+\epsilon}$. This completes the proof of the claim.


\bibliographystyle{alpha}
\bibliography{reference}

\end{document}